\documentclass{SSMH_nonOA}

\usepackage[numbers]{natbib}

  \usepackage{graphicx}%
\usepackage{multirow}%
\usepackage{amsmath,amssymb,amsfonts}%
\usepackage{amsthm}%
\usepackage{mathrsfs}%
\usepackage[title]{appendix}%
\usepackage{xcolor}%
\usepackage{textcomp}%
\usepackage{manyfoot}%
\usepackage{booktabs}%
\usepackage{algorithm}%
\usepackage{algorithmicx}%
\usepackage{algpseudocode}%
\usepackage{listings}%
\usepackage{graphicx}
\usepackage{amsthm,amssymb,amsmath}
\usepackage{subcaption}
\usepackage{caption}
\usepackage{float}
\usepackage{mathtools}

\usepackage[colorlinks]{hyperref}
\hypersetup{
  colorlinks,
  citecolor=coolblack,
  linkcolor=coolblack,
  urlcolor=coolblack}
\DeclarePairedDelimiter\ceil{\lceil}{\rceil}
\DeclarePairedDelimiter\floor{\lfloor}{\rfloor}
\newcommand{\articletype}{Original Research Paper}
\newcommand{\received}{}
\newcommand{\accepted}{}
\newcommand{\DOI}{}
\newcommand{\biblinfo}{}
\newcommand{\HandlingEditor}{Handling Editor}

\usepackage{lipsum}

\newtheorem{conjecture}[proposition]{Conjecture}

 \newtheorem*{definition*}{Definition}

\def\R{\mathbb{R}}

\begin{document}

\title{Further analysis of Peeling Sequences}


\author{\large{Dániel Gábor Simon}$^{1,2}$} 

\affiliation[1]{HUN-REN Alfréd Rényi Institute of Mathematics, Budapest, Hungary}
\affiliation[2]{Eötvös Loránd University, Budapest, Hungary}
\newcommand{\CorrespondingAuthorEmail}{dgsimon@renyi.hu} 

\newcommand{\KeyWords}{Peeling, algorithm, convexity, geometry, combinatorics}

\newcommand{\PrimarySubjectClass}{52C35}
\newcommand{\SecondarySubjectClass}{52C45}

\begin{abstract}
\linespread{1}\selectfont
 \noindent Let $P\subset \R^2$ be a set of $n$ points in general position. A peeling sequence of $P$ is a list of its points, such that if we remove the points from $P$ in that order, we always remove the next point from the convex hull of the remainder of $P$. Using the methodology of Dumitrescu and Tóth \cite{Dumitrescu}, with a more careful analysis, we improve the upper bound on the minimum number of peeling sequences for an $n$ point set in the plane from $12.29^n/100$ to $9.78^n/500$.
\end{abstract}

\maketitle
\section{Introduction}
A set of points  $P\subset\R^2$ lies in general position if no three points of $P$ are collinear. Let $P$ be such a set. We define a peeling algorithm in the following way: Take the convex hull of the current point set, remove one point from the hull, and write down its label. Repeat on the set without that point. Stop when there are no more points left in the set.

This way, we generated a sequence of the labels of the vertices. It is called a \emph{peeling sequence}. The set of all peeling sequences consists of all orderings of the labels that can be achieved by the above algorithm.

\begin{definition*}
Given a point set $P$ in general position in $\R^2$, let $g(P)$ denote the number of peeling sequences of $P$. 

Let 
$g(n)=\min_{|P|=n} g(P)$, where the minimum is taken over all $n$-element point sets $P$ in general position in $\R^2$.
\end{definition*}

Peeling sequences in this form were first introduced by Dumitrescu \cite{dum1}. The currently known best bounds are $\Omega(3^n)\leq g(n)\leq O(12.29^n)$ by Dumitrescu \cite{dum1}, and Dumitrescu and Tóth \cite{Dumitrescu} respectively. In this paper, we improve the upper bound to $g(n)\leq O(9.78^n)$.

\begin{theorem}\label{maintheorem}
 Let $g(n)$ be the minimum number of peeling sequences a set of $n$ points can have. For any $n\geq 6$,\\$$g(n)\leq\frac{9.78^n}{500}.$$   
\end{theorem}

The question was also investigated in higher dimensions. The interested reader can find the current best upper bounds in the author's other paper \cite{simon}.

\smallskip

   \textbf{Related work.} Many previously investigated properties of point sets are connected to the convexity of the set. In our example, the number of peeling sequences is maximal for a convex $n$-gon, and minimal for some highly nonconvex structure. Similar parameters connected to the convexity are the number of \emph{triangulations} (see \cite{triang3, triang2, triang1} for some bounds), the number of \emph{polygonizations} (see \cite{poly1}, \cite{poly2}), the number of \emph{non-crossing perfect matchings} (see \cite{matching1}), etc.
   
   Peelings of sets of points have also appeared in the literature before. Many results are about \emph{convex-layer peeling} \cite{even1, even2, grid1}, where at each step all vertices of the current convex hull are being removed, and the total number of layers is of interest. Such methods have been used in the proof of the famous empty hexagon theorem \cite{gerken}\cite{nicolas}; see also \cite{dumremark}. Computer graphics also uses different peeling algorithms, the affine curve shortening \cite{alvarez, grid2, sapiro} is an interesting example.  

    \section{Construction}

    The construction is identical to that defined by Dumitrescu and Tóth \cite{Dumitrescu}. We describe it here for self-containment purposes and to facilitate understanding of the other types of sets considered in this paper.

    We define the sets $S_n$ recursively for all $n\in\mathbb{N}$. Each set $S_n$ contains $n$ points, and is "flat": all points are very close to a line compared to the minimum pairwise distance in $S_n$. 
    
    \textbf{Base cases:}
    Let $S_1=\{\mathbf{o_1}\}$ be a single point, $S_2=\{\mathbf{o_2},\mathbf{u}\}$ be two distinct points on the $x$ axis.

    \textbf{Recursive step:} For $n\geq3$ we construct $S_n$ as follows. Let $n=n_1+n_2+n_3$ where each $n_i$ is either $\floor{\frac{n}{3}}$ or $\ceil{\frac{n}{3}}$, and $n_1\geq n_2\geq n_3$. Assume recursively that $S_{n_1},S_{n_2},S_{n_3}$ have already been defined. 
    
    \textbf{Placement of the sets:} Let $\mathbf{o_n}$ be a point serving as the origin of $S_n$. Let $r_1, r_2, r_3$ be three rays emanating from  $\mathbf{o_n}$ with angles of $120\degree$ between consecutive rays in clockwise order. Take $r_1$ to lie along the positive $x$-axis. 
    For each $i = 1,2,3$, place a rotated and translated copy of $S_{n_i}$ so that its first ray's direction aligns with ray $r_i$, as illustrated in Figure~\ref{fig:S'_n}. Furthermore, orient each copy so that its ray $r_1$ lies farthest from the origin; this makes the construction unique. Each set is far enough from the origin such that the convex hull consists of exactly one vertex from each ray (provided that the rays are non-empty). Also, the projections of the three sets to the $x$ axis are disjoint, so any point of $S_{n_1}$ has strictly larger $x$-coordinate than any point of $S_{n_2}$, and any point of $S_{n_3}$ has a strictly smaller $x$-coordinate than points of the other sets.

\begin{figure}
	\begin{center}
		\includegraphics[scale=0.25]{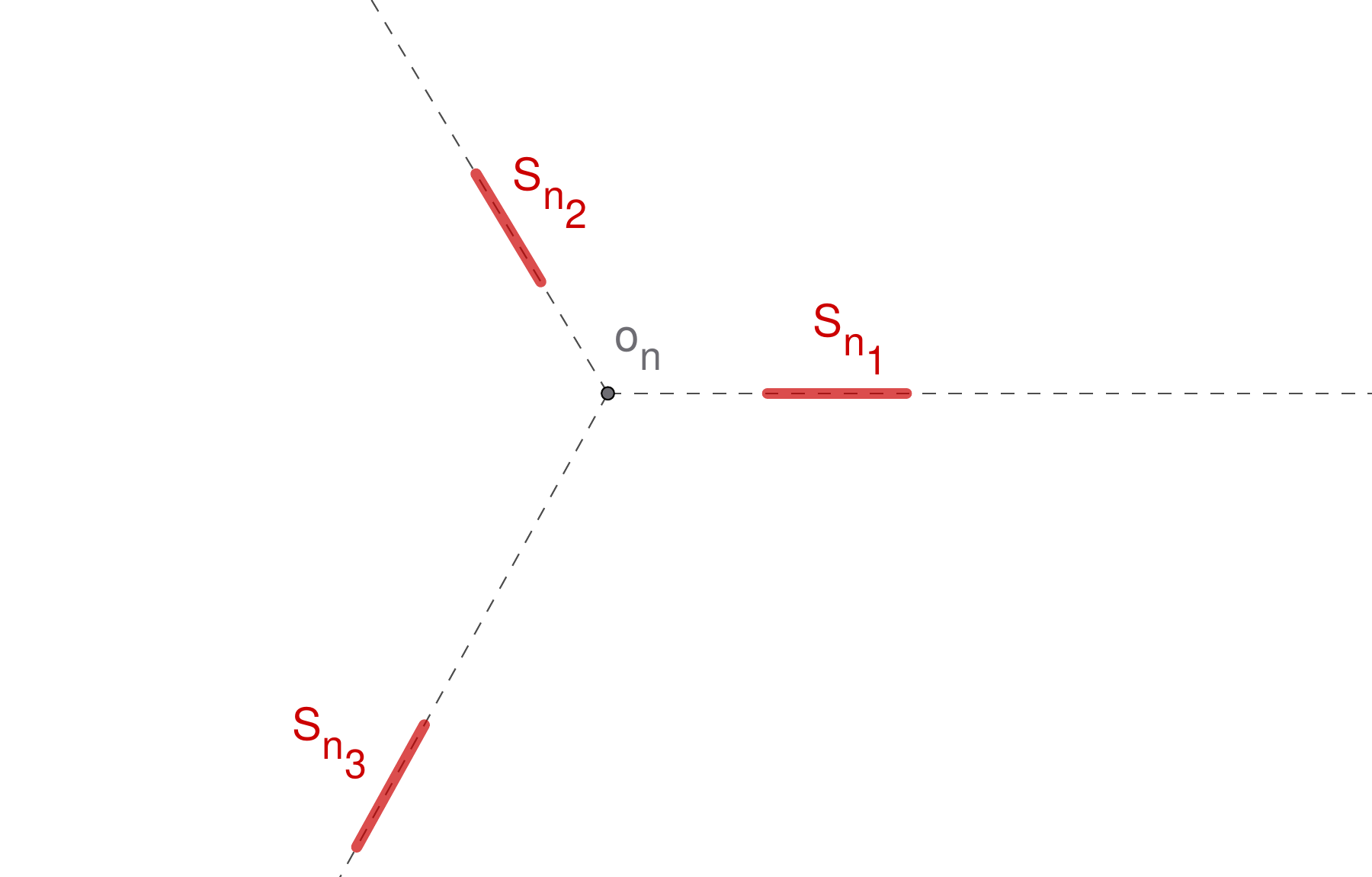}
		\caption{$S'_n$ before the flattening step}
		\label{fig:S'_n}
	\end{center}
\end{figure}
\begin{figure}
	\begin{center}
		\includegraphics[scale=0.25]{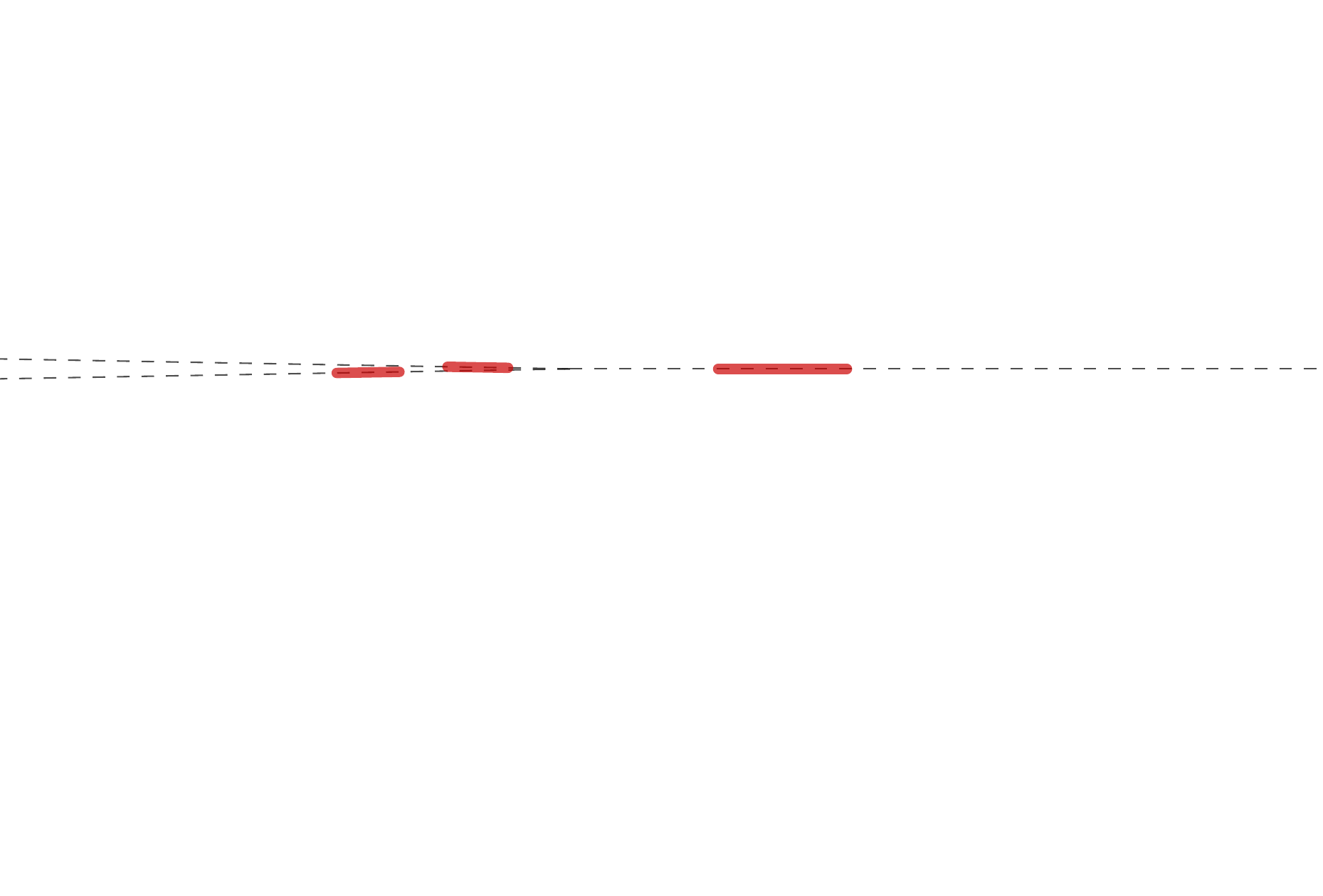}
		\caption{$S_n$ after the flattening step}
		\label{fig:S_n}
	\end{center}
\end{figure}

    \textbf{Flattening:} Let $S_n'$ denote the union of the three sets constructed above. We obtain $S_n$ by applying the transformation $(x,y)\rightarrow(x,\epsilon y)$ for a sufficiently small number $\epsilon>0$. For illustration, see Figure \ref{fig:S_n}.

   We also define a related family of point sets, which will be used in the proof of the upper bound.

    Let $B_n$ be the point set obtained from $S_n$ by selecting ray $r_1$ of $S_n$ and deleting its two subrays that are farthest from $\mathbf{o_n}$. Here, a \emph{subray} refers to a ray belonging to a copy of some $S_{n_i}$, where this copy is placed along the $i$-th ray $r_i$. So $B_n$ contains approximately $\frac{7n}{9}$ points. See Figure~\ref{fig:B_n} for an illustration.
    

    \begin{figure}
	\begin{center}
		\includegraphics[scale=0.25]{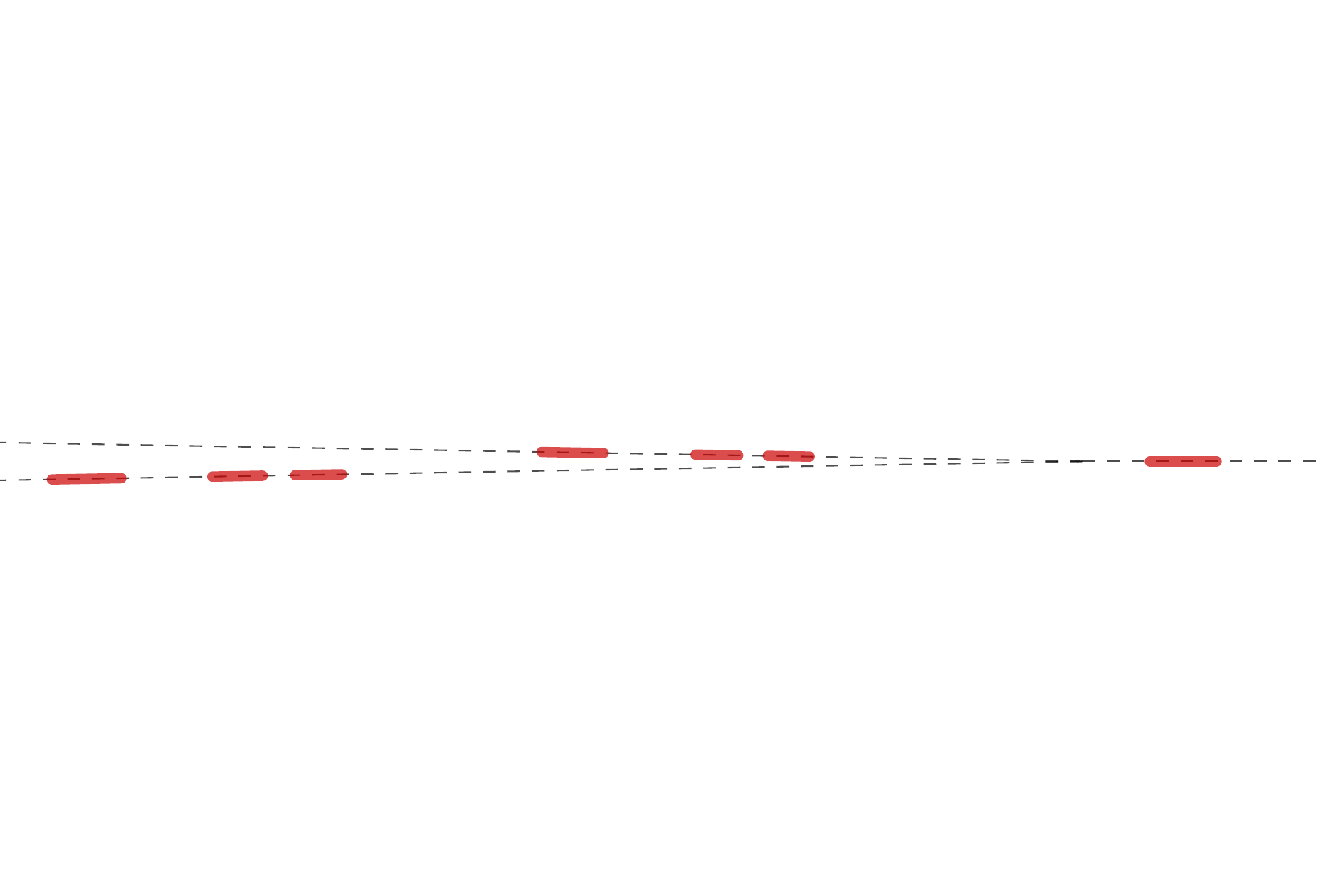}
		\caption{The construction of $B_n$. On the ray of the positive $x$-axis, we remove the $2$ outside subrays of the corresponding recursive set from $S_n$.}
		\label{fig:B_n}
	\end{center}
\end{figure}

    \section{Estimation Lemmas}
    The following statements are required for the proof of Theorem \ref{maintheorem}. They were also used in the analysis of Dumitrescu and Tóth \cite{Dumitrescu}.

\begin{definition*}
    For any $0\leq p\leq 1$, define \[ H(p)=-p\log_2 p-(1-p)\log_2{(1-p)} \] the \emph{binary entropy function}.
\end{definition*}
With this function, we can use the following estimate from \cite{Mitzenmacher}, Corollary 10.3.
    \begin{lemma}
        For any integer $n\geq1$, and $0\leq\alpha\leq\tfrac12$, \\ $$\binom{n}{\alpha n}\leq 2^{nH(\alpha)}$$
    \end{lemma}

    The next lemma is a simple estimate.
    
\begin{lemma}\label{lem:l}
    For any positive integers $n,k,l$ satisfying $\frac{n}{l+1}\leq k\leq n$,\\ $$\binom{n}{\ceil{k}}\leq l\binom{n}{\floor{k}}.$$
    
\end{lemma}

 The proofs of the next two lemmata are also discussed in \cite{Dumitrescu} Section 3.2. The first follows from a simple induction argument.
 
\begin{lemma}
For all integers $n\geq 36$ and $0\leq k\leq n$, $$\binom{n}{k}\leq \frac{2^n}{7}$$
Moreover, $$\binom{n}{k}\leq \frac{2^n}{2}$$ for all positive integers $n$.
\end{lemma}
\begin{proof}
    It is easy to verify the statement for $n=36$. For larger $n$, observe that $$\binom{n}{k}=\binom{n-1}{k-1}+\binom{n-1}{k}\leq2\times\frac{2^{n-1}}{7}=\frac{2^n}{7},$$ which completes the induction.
    The moreover part holds for $n=1$, and the same inductive argument applies.
\end{proof}

To prove the next lemma, we need to define \emph{simplified peeling sequences}.

\begin{definition*}
    Let $Z=X_1\cup X_2\cup \dots \cup X_k$ where the sets $X_i$ are pairwise disjoint. Let $\pi$ be a peeling sequence of $Z$. Replace each point in $\pi$ with the symbol $i$, where $i$ is the index of the set $X_i$ containing the point. The resulting sequence $\pi^*$ of indices is called a \emph{simplified peeling sequence} of $Z$ with respect to this partition.
\end{definition*}

\begin{lemma}\label{decomp}\textbf{[Dumitrescu, Tóth]}
    Let $Z = X_1\cup X_2\cup\dots\cup X_k$ be a set of points in general position, such that for all $i\neq j$ $X_i\cap X_j =\emptyset$, $|X_i| =$ $n_i$, and $|Z| =\sum_{i=1}^kn_i= n$. Then $$g(Z)\leq \binom{n}{n_1,n_2,n_3,\dots,n_k}\prod_{i=1}^kg(X_i).$$
\end{lemma}

\begin{proof}
    We first estimate how many peeling sequences correspond to the same simplified peeling sequence $\pi^*$.

    Fix an index $i$. Restricting $\pi$ to the positions where $\pi^*$ equals $i$ yields a subsequence $\pi_i$. This subsequence must be a peeling sequence of $X_i$, as otherwise those points could not have been removed in this order even in the absence of other points. Thus, for each $i$, there are at most $g(X_i)$ possible choices for $\pi_i$. Therefore, the number of peeling sequences corresponding to a fixed $\pi^*$ is at most $$\prod_{i=1}^k g(X_i).$$

    To conclude the proof, it is enough to count the number of distinct simplified peeling sequences.
    Each such sequence consists of exactly $n_i$ occurrences of the symbol $i$, for each $i=1,\dots, k$. Hence, the total number of such sequences is at most $$\binom{n}{n_1,n_2,n_3,\dots,n_k}=\frac{n!}{n_1!n_2!n_3!\dots n_k!}.$$ Multiplying the two upper bounds concludes the proof.
    \end{proof}
    
    \section{Proof of the main Theorem}

We prove the following theorem, from which Theorem \ref{maintheorem} follows.
\\
\begin{theorem}
For all $n \geq 6$,
\[
g(S_n) \leq 9.78^n/500
\quad \text{and} \quad
g(B_n) \leq 8.67^n/500.
\]
\end{theorem}
\begin{proof}

We proceed by induction on $n$.

For the base cases, we verify the statement for $6\leq n\leq 53$, since all larger instances are constructed from these.

If we construct $S_n$ and $B_n$ as discussed before, they satisfy the required bounds for these values of $n$. Note that in the first $\frac{n}{3}-1$ steps of peeling $S_n$, the convex hull consists of $3$ vertices. This yields the bound $$g(S_n)\leq 3^{\lfloor\frac{n}{3}\rfloor}\left\lceil\frac{2n}{3}\right\rceil!$$ which is already sufficient in the given range of $n$.

For $B_n$, only the first $\lfloor\frac{n}{9}\rfloor$ steps guarantee exactly $3$ vertices on the convex hull. This gives $$g(B_n)\leq 3^{\lfloor\frac{n}{9}\rfloor}\left\lceil\frac{2n}{3}\right\rceil!,$$ which is also sufficient in the required range of $n$.

For $n\geq 54$, we conclude by induction. Suppose the statement holds for all $$\lfloor\frac{n}{9}\rfloor\leq k\leq n-1.$$ That is, $S_k$ has at most $a^k/500$ peeling sequences, and $B_k$ has at most $b^k/500$, where $a=9.78$, and $b=8.67$.

We analyse the behaviour of each $B_n$ and $S_n$, based on the first step at which an entire ray disappears. Depending on how many vertices have been removed by that point, we can deduce some information about the remaining two other rays by the pigeonhole principle. By induction, we can bound the number of peeling sequences on these rays, and combine them with Lemma \ref{decomp}.

Using the sets $B_n$ has the purpose of making these pigeonhole principle arguments more refined. In particular, it enables us to apply stronger bounds to any ray that has lost at least a $2/9$ fraction of its vertices. Ultimately, this leads to a stronger bound than only analysing the $S_{n}$-s.

\subsection{$B_n$}

We begin with the analysis of $B_n$. Throughout this subsection, we assume inductively that the theorem holds for all sets up to $S_{n-1}$ and $B_{n-1}$. We split the peeling sequences into classes according to the index of the last element whose removal made one of the $3$ main rays to disappear. We call this index $s$. We refer to Figure \ref{fig:b} for an example value of $s$. 

In all the following cases, unless stated otherwise, we assume that the first ray to disappear during peeling is the one that initially contains only a single subray. Let $\sigma_i$ denote the total number of peeling sequences in case $i$, for $1\leq i\leq6$. We show that for each $i$, $$\sigma_i\leq C_ib^n$$ for suitable constants $C_i$, where $\sum_{i=1}^6C_i<1/500$.
Then we can conclude that $$\sigma=\sum_{i=1}^6 \sigma_i\leq\sum_{i=1}^6 C_ib^n\leq b^n/500,$$ proving the first half of the theorem.\\

\begin{figure}
	\begin{center}
		\includegraphics[scale=0.25]{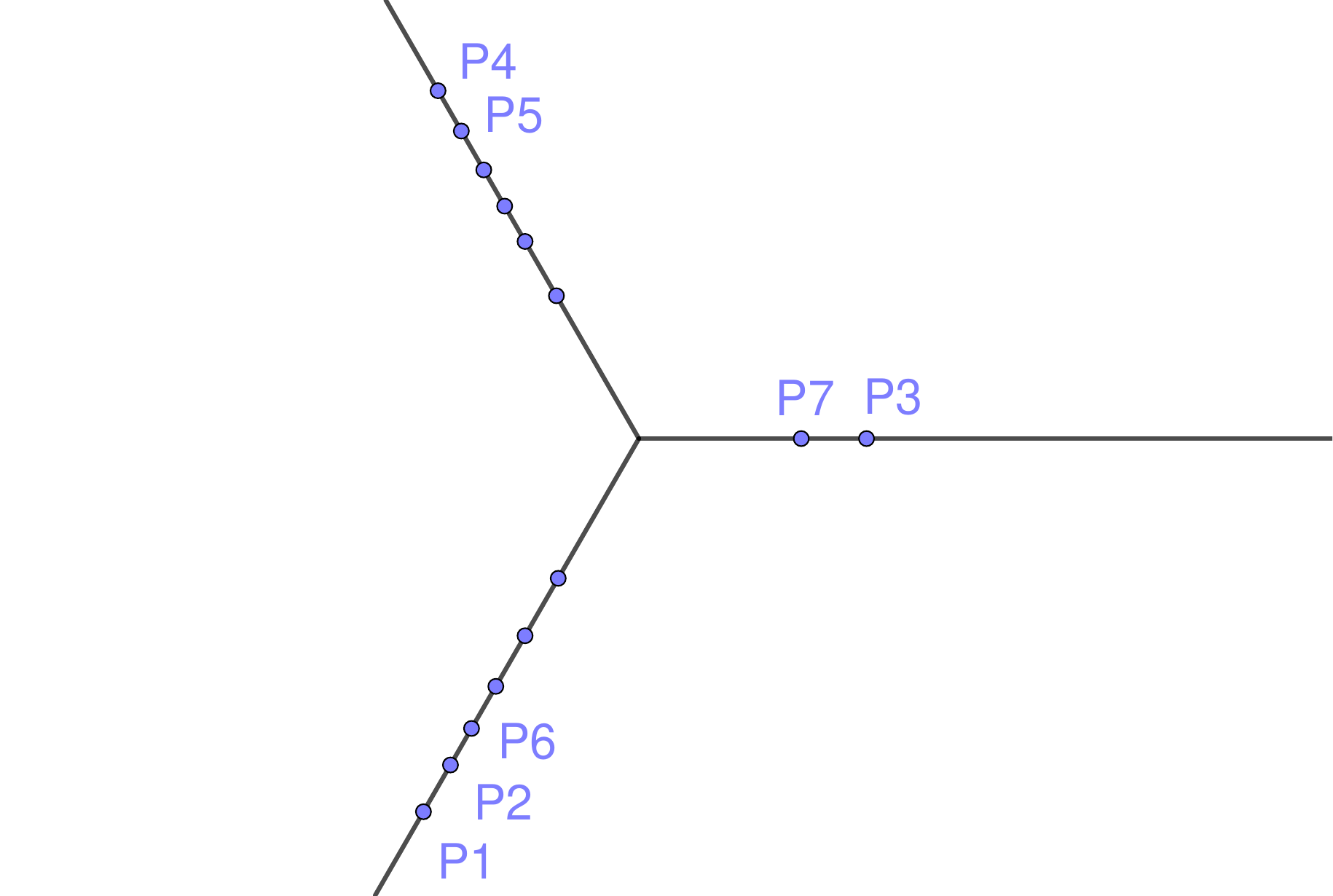}
		\caption{An unflattened representation of $B_{18}$, with a labeling of the starting elements of a peeling sequence. For simplicity, we place vertices on the same ray as collinear, but they are just close to it. Here, the first ray disappears after the $7$th removed element, hence $s=7$ would hold for this sequence.}
		\label{fig:b}
	\end{center}
\end{figure}

    \textbf{Case 1} $\floor{n/9}\leq s\leq \ceil{7n/27}.$

    Let $T_1$ be the set of all simplified peeling sequences on the tripartition defined by the $3$ rays, for which the smallest ray disappears at an index $s$ in the above interval.

    Now we estimate $|T_1|$. The small ray must disappear within the first at most $\ceil{7n/27}$ steps, so their indices can be in at most $\binom{\ceil{7n/27}}{\ceil{n/9}}$ allocations. For each such allocation, the remaining two rays can be placed in at most $\binom{\ceil{2n/3}}{\ceil{n/3}}$ ways.
   
    Then we can bound the size of $T_1$ by the use of our estimation lemmas in section $4$, and by using Lemma \ref{lem:l} with $l=2$.
\begin{equation*}
\begin{split}   
    |T_1|\leq \binom{\ceil{7n/27}}{\ceil{n/9}}\cdot \binom{\ceil{2n/3}}{\floor{n/3}}&\leq 2\binom{\ceil{7n/27}}{\floor{\frac37\ceil{7n/27}}}\cdot \binom{\ceil{2n/3}}{\floor{n/3}}\\&\leq 8\cdot 2^{7H(3/7)n/27+2n/3}/7\leq2^{0.9221n}\cdot8/7
\end{split} 
\end{equation*}

   For each $\pi^* \in T_1$, the order of removals up to step $s$ is fixed.
    Hence, the smallest ray's element order is determined, and the possible orderings of the other elements are bounded above by $g(S_{\ceil{\frac{n}{3}}})g(S_{\ceil{\frac{n}{3}}})$.
    Hence, by using induction, the total number of such peeling sequences is at most:

    $$\sigma_1\leq |T_1|\cdot g(S_{\ceil{n/3}})^2\leq(2^{0.9221n}\cdot8/7)\cdot (a^{2n/3}/250000)\cdot a^{4/3}\leq C_1 \cdot b^n$$

    This inequality holds provided that

$$b^n\geq a^{2n/3}\cdot 2^{0.9221n}$$  $$8\cdot a^{2/3}/1750000=C_1.$$
\\

    \textbf{Case 2} $\ceil{7n/27}\leq s\leq \ceil{8n/27}$
    
    In this range, by the pigeonhole principle, one of the other two rays must have turned into a subset of $B_{\ceil{n/3}}$ by step $s$.

    We estimate $T_2$:

    $$|T_2|\leq \binom{\ceil{8n/27}}{\ceil{n/9}}\cdot \binom{\ceil{2n/3}}{\floor{n/3}}\leq 8\cdot 2^{8H(3/8)n/27+2n/3}/7\leq2^{0.9495n}\cdot8/7$$

By induction and a similar argument to the previous case, the total number of such peeling sequences is at most:
  \begin{equation*}
    \begin{split}
     \sigma_2&\leq |T_2|\cdot g(S_{\ceil{n/3}})\cdot g(B_{\ceil{n/3}})\\ &\leq(2^{0.9495n}\cdot8/7)\cdot (a^{n/3}/500)\cdot (b^{n/3}/500)\cdot a^{2/3}\cdot b^{2/3}\\ &\leq C_2 \cdot b^n
    \end{split}
\end{equation*}

     This inequality holds if

$$b^n\geq a^{n/2}\cdot 2^{1.4243n}$$ $$8\cdot a^{2/3}\cdot b^{2/3}/1750000=C_2.$$
\\
    
 \textbf{Case 3} $\ceil{8n/27}\leq s\leq \ceil{12n/27}$
 
     By the pigeonhole principle, now we know that in the first $s$ steps, either both of the remaining rays have turned into a subset of a $B_{\ceil{n/3}}$, or one ray has lost an entire subray.

Using Lemma~\ref{lem:l} with $l = 3$, we obtain

\begin{equation*}
\begin{split}
    |T_3|\leq \binom{\ceil{12n/27}}{\ceil{n/9}}\cdot \binom{\ceil{2n/3}}{\floor{n/3}}&\leq 3\cdot\binom{\ceil{12n/27}}{\floor{\frac{3}{12}\ceil{12n/27}}}\cdot \binom{\ceil{2n/3}}{\floor{n/3}}\\&\leq12\cdot 2^{12\cdot H(3/12)n/27+2n/3}/7\\&\leq2^{1.0273n}\cdot12/7
\end{split}
\end{equation*}

If a ray loses a subray, it is a subset of two copies of $S_{\ceil{n/9}}$. An upper bound on the number of ways the remainder of this ray can be peeled follows by Lemma \ref{decomp}, showing it is at most $\binom{2\ceil{n/9}}{\ceil{n/9}}\cdot g(S_{\ceil{n/9}})^2\leq 2\cdot2^{2n/9}\cdot g(S_{\ceil{n/9}})^2$. Hence,

\begin{equation*}
\begin{split}
       \sigma_3&\leq |T_3|\cdot max[( g(B_{\ceil{n/3}})\cdot g(B_{\ceil{n/3}}),g(S_{\ceil{n/3}})\cdot2\cdot 2^{2n/9}\cdot g(S_{\ceil{n/9}})^2] \\ &\leq(2^{1.0273n}\cdot12/7)\cdot max[(b^{n/3}/500)^2\cdot b^{4/3},\\ &(a^{n/3}/500)\cdot a^{2/3}\cdot2^{2n/9+1}\cdot (a^{n/9}/500)^2\cdot a^{16/9}]\\ &\leq C_3\cdot b^n
\end{split}
\end{equation*}
     This inequality holds if

$$b^n\geq max(2^{3.0819n},a^{5n/9}\cdot 2^{1.2496n})$$ $$max(12\cdot b^{4/3}/1750000,24\cdot a^{22/9}/875000000) =C_3.$$
\\

    \textbf{Case 4}     $\ceil{12n/27}\leq s\leq \ceil{15n/27}$
    
        By the pigeonhole principle, at least two additional subrays have disappeared until step $s$. This can happen in two ways: Either one ray has lost $2$ subrays so it became a subset of an $S_{\ceil{n/9}}$, or both rays lost $1$-$1$ subrays, each becoming a subset of a pair of $S_{\ceil{n/9}}$'s.

        Using $l=4$ in the application of Lemma \ref{lem:l} we obtain

    $$|T_4|\leq \binom{\ceil{15n/27}}{\ceil{n/9}}\cdot \binom{\ceil{2n/3}}{\floor{n/3}}\leq 16\cdot 2^{15\cdot H(3/15)n/27+2n/3}/7\leq2^{1.0678n}\cdot16/7.$$

If a ray loses a subray, it is a subset of two copies of $S_{\ceil{n/9}}$. An upper bound on the number of ways the remainder of this ray can be peeled follows by Lemma \ref{decomp}, showing it is at most $\binom{2\ceil{n/9}}{\ceil{n/9}}\cdot g(S_{\ceil{n/9}})^2\leq 2\cdot2^{2n/9}\cdot g(S_{\ceil{n/9}})^2$. Hence,
\begin{equation*}
\begin{split}
       \sigma_4 & \leq |T_4|\cdot max[(g(S_{\ceil{n/9}})\cdot g(S_{\ceil{n/3}}),4\cdot2^{4n/9}g(S_{\ceil{n/9}})^4)] \\ &\leq (2^{1.0678n}\cdot16/7)\cdot max[(a^{n/9}/500)\cdot (a^{n/3}/500)\cdot a^{14/9},\\ &2^{4n/9}\cdot (a^{n/9}/500)^4\cdot4\cdot a^{32/9}] \\ &\leq C_4\cdot b^n
\end{split}
\end{equation*}

     Since the second part of the maximum is strictly larger for $n\geq54$, this inequality holds when:

$$b^n\geq a^{4n/9}\cdot 2^{1.5123n}$$ $$64a^{32/9}/437500000000=C_{4}.$$
\\

    \textbf{Case 5}     $\ceil{15n/27}\leq s\leq \ceil{21n/27}$
    
  By the pigeonhole principle, at least two extra subrays disappeared from the construction until step $s$. This can happen in only one way: one ray is a subset of $S_{\ceil{n/9}}$, and the other one is a subset of a pair of $S_{\ceil{n/9}}$.

  We use $l=6$ when applying Lemma \ref{lem:l}.
        
    $$|T_5|\leq \binom{\ceil{21n/27}}{\ceil{n/9}}\cdot \binom{\ceil{2n/3}}{\floor{n/3}}\leq 24\cdot 2^{21\cdot H(3/21)n/27+2n/3}/7\leq 2^{1.127n}\cdot24/7$$

The total number of such sequences is at most:
\begin{equation*}
    \begin{split}
    \sigma_5&\leq |T_5|\cdot g(S_{\ceil{n/9}})\cdot2\cdot2^{2n/9}\cdot g(S_{\ceil{n/9}})^2\\ &\leq(2^{1.127n}\cdot24/7)\cdot (a^{n/9}/500)\cdot2^{2n/9}\cdot  (a^{n/9}/500)^2\cdot a^{24/9}\cdot2\leq C_5\cdot b^n  
    \end{split}
\end{equation*}

     This inequality holds if

$$b^n\geq a^{n/3}\cdot 2^{1.3493n}$$ $$48\cdot a^{24/9}/875000000=C_{5}.$$
\\

    \textbf{Case 6} $\floor{9n/27}\leq s\leq \ceil{21n/27}$

    In this case, the first ray to disappear is a longer $S_{\ceil{n/3}}$ ray, not the short ray.

    We count with all the possible $s$ values at once and give an upper bound on the number of simplified peeling sequences again, similar to the other cases. Here we can take $l=2$ again.

    $$|T_6|\leq \binom{\ceil{21n/27}}{\ceil{n/3}}\cdot \binom{\ceil{12n/27}}{\ceil{n/9}}\leq 24\cdot 2^{21\cdot H(9/21)n/27+4n/9}/7\leq2^{1.2108n}\cdot24/7$$

    The total number of such sequences is at most:

    $$\sigma_6\leq |T_6|\cdot(g(S_{\ceil{n/3}})\cdot g(S_{\ceil{n/9}}))\leq (2^{1.2108n}\cdot24/7)\cdot (a^{n/3}/500)\cdot (a^{n/9}/500)\cdot a^{14/9}\leq C_6\cdot b^n$$

    This inequality holds if 

    $$b^n\geq a^{4n/9}\cdot 2^{1.2108n}$$ $$24\cdot a^{14/9}/1750000=C_{6}.$$
\\

    After all that case analysis, we now collect the conditions we have on the value of $b$, and on the values of the coefficients, in order to make the induction work. If they are all fulfilled by $a=9.78$, $b=8.67$, the proof is done on $B_n$.

So $b$ must be at least as large as any of the following:

$b^n\geq 8.67^n\geq a^{2n/3}\cdot 2^{0.9221n}$ 

$b^n\geq 8.40^n\geq a^{n/2}\cdot 2^{1.4243n}$ 

$b^n\geq 8.47^n\geq max(2^{3.0819n},a^{5n/9}\cdot 2^{1.2496n})$ 

$b^n\geq 7.86^n\geq a^{4n/9}\cdot 2^{1.5123n}$

$b^n\geq 5.45^n\geq a^{n/3}\cdot 2^{1.3493n}$

$b^n\geq 6.38^n\geq a^{4n/9}\cdot 2^{1.2108n}$

We can clearly see that the conditions hold for $b=8.67$.

We also calculate the coefficients:

$C_1=8\cdot a^{2/3}/1750000\leq 0.00003$

$C_2=8\cdot a^{2/3}\cdot b^{2/3}/1750000\leq 0.00009$

$C_3=max(12\cdot b^{4/3}/1750000,24\cdot a^{22/9}/875000000)\leq 0.00013$

$C_4=64\cdot a^{32/9}/437500000000\leq 0.00001$

$C_5=48\cdot a^{24/9}/875000000\leq 0.00003$

$C_6=24\cdot a^{14/9}/1750000\leq 0.00048$

Since $\sum_{i=1}^6C_i\leq 0.00077\leq 1/500$, the induction for $g(B_n)$ follows. 

It remains to analyze the number of peeling sequences of $S_n$.

\subsection{${S_n}$}

Now we perform the same type of analysis for $a$ and the sets $S_n$. As before, we always assume by induction that the theorem holds up until $B_{n-1}$ and $S_{n-1}$.

As previously, let $s$ denote the index of the last element of the first disappearing ray in the modified peeling sequence.\\

\textbf{Case 1}
$\floor{9n/27}\leq s\leq \ceil{13n/27}$

Using $l=2$ in Lemma \ref{lem:l} we obtain
$$|T_1|\leq 3\cdot \binom{\ceil{13n/27}}{\floor{n/3}}\cdot \binom{\ceil{2n/3}}{\floor{n/3}}\leq 24\cdot 2^{13\cdot H(4/13)n/27+2n/3}/7\leq24\cdot 2^{1.0955n}/7$$

The number of peeling sequences on each remaining rays can be upper bounded by $g(S_{\ceil{n/3}})$. Thus, the total number of sequences is at most:

    $$\sigma_1\leq |T_1|\cdot(g(S_{\ceil{n/3}})^2\leq (24\cdot 2^{1.0955n}/7)\cdot (a^{n/3}/500)^2\cdot a^{4/3}\leq C_1\cdot a^n$$

    The last inequality holds if

    $$a^{n}\geq 2^{3.2865n}$$ $$24\cdot a^{4/3}/1750000=C_1.$$ \\
    
\textbf{Case 2}
$\ceil{13n/27}\leq s\leq \ceil{14n/27}$

By the pigeonhole principle, at least one of the remaining rays got reduced to a subset of a $B_{\ceil{n/3}}$.

$$|T_2|\leq 3\cdot\binom{\ceil{14n/27}}{\floor{n/3}}\cdot \binom{\ceil{2n/3}}{\floor{n/3}}\leq 24\cdot 2^{14\cdot H(5/14)n/27+2n/3}/7\leq24\cdot 2^{1.1543n}/7$$

Thus,
\begin{equation*}
    \begin{split}
    \sigma_2&\leq |T_2|\cdot g(S_{\ceil{n/3}})\cdot g(B_{\ceil{n/3}})\\ &\leq (24\cdot 2^{1.1543n}/7)\cdot (a^{n/3}/500)\cdot (b^{n/3}/500)\cdot b^{2/3}\cdot a^{2/3}\\ &\leq C_2\cdot a^n
   \end{split}
\end{equation*}
    The last inequality holds if

$$a^n\geq 2^{1.7315n}\cdot b^{n/2}$$  $$24\cdot a^{2/3}\cdot b^{2/3}/1750000=C_2.$$\\
    
\textbf{Case 3}
$\ceil{14n/27}\leq s\leq \ceil{15n/27}$

Either both remaining rays got reduced to a subset of $B_{\ceil{n/3}}$, or one of them lost an entire subray and became a subset of a union of two $S_{\ceil{n/9}}$'s.

$$|T_3|\leq 3\cdot\binom{\ceil{15n/27}}{\floor{n/3}}\cdot \binom{\ceil{2n/3}}{\floor{n/3}}\leq 24\cdot 2^{15\cdot H(6/15)n/27+2n/3}/7\leq24\cdot 2^{1.2061n}/7$$

By Lemma \ref{decomp}, the number of ways a pair of $S_{\ceil{n/9}}$ can be peeled is at most $\binom{2\ceil{n/9}}{\ceil{n/9}}\cdot g(S_{\ceil{n/9}})^2\leq 2\cdot2^{2n/9}\cdot g(S_{\ceil{n/9}})^2$.

Thus,
\begin{equation*}
    \begin{split}
         \sigma_3&\leq |T_3|\cdot max[g(S_{\ceil{n/3}})\cdot2\cdot2^{2n/9}\cdot g(S_{\ceil{n/9}})^2,g(B_{\ceil{n/3}})^2]\\ &\leq (24\cdot 2^{1.2061n}/7)\cdot max[(a^{n/3}/500)\cdot2^{2n/9}(a^{n/9}/500)^2\cdot 2\cdot a^{22/9},\\ &(b^{n/3}/500)^2\cdot b^{4/3}]\\ &\leq C_3\cdot a^n
    \end{split}
\end{equation*}

    The last inequality holds if

    $$a^n\geq max(2^{3.2138n},2^{1.2061n}\cdot b^{2n/3})$$ $$max(48\cdot a^{22/9}/875000000,24\cdot b^{4/3}/1750000)=C_3.$$\\
    
\textbf{Case 4}
$\ceil{15n/27}\leq s\leq \ceil{17n/27}$

At least one of the remaining rays lost a subray, say ray $A$. In addition to that, either ray $A$ misses another sub-subray (a ray of a subray), or the other ray is a subset of $B_{\ceil{n/3}}$.

$$|T_4|\leq 3\cdot \binom{\ceil{17n/27}}{\floor{n/3}}\cdot \binom{\ceil{2n/3}}{\floor{n/3}}\leq 24\cdot 2^{17\cdot H(8/17)n/27+2n/3}/7\leq24\cdot 2^{1.2948n}/7$$

When the two rays are a subset of a $B_{\ceil{n/3}}$ and a subset of a pair of $S_{\ceil{n/9}}$ respectively, we already know what estimates we can use.

When one ray is a subset of the union of a copy of $S_{\ceil{n/9}}$ and two copies of $S_{\ceil{n/27}}$, by Lemma \ref{decomp} and the estimation lemmas,  we get an upper bound of $\binom{\ceil{3n/27}+2\ceil{n/27}}{\floor{{3n/27}}}\binom{2\ceil{n/27}}{\ceil{n/27}}\cdot g(S_{\ceil{n/9}})\cdot g(S_{\ceil{n/27}}^2)\leq 8\cdot 2^{7n/27}\cdot g(S_{\ceil{n/9}})\cdot g(S_{\ceil{n/27}})^2$ for the number of peeling sequences of the ray.

Thus,
\begin{equation*}
    \begin{split}
        \sigma_4&\leq |T_4|\cdot max[g(S_{\ceil{n/3}})\cdot2^{7n/27+3}\cdot g(S_{\ceil{n/9}})\cdot g(S_{\ceil{n/27}})^2,\\ &g(B_{\ceil{n/3}})\cdot2\cdot2^{2n/9}\cdot g(S_{\ceil{n/9}})^2] \\ &\leq max[(192\cdot 2^{1.2948n}/7)\cdot (a^{n/3}/500)\cdot 2^{7n/27}\cdot (a^{n/9}/500)\cdot (a^{n/27}/500)^2\cdot a^{98/27}, \\ & (b^{n/3}/500)\cdot 2^{2n/9}\cdot (a^{n/9}/500)^2\cdot2\cdot a^{16/9}\cdot  b^{2/3}]\\ &\leq C_4\cdot a^n
    \end{split}
\end{equation*}

    The last inequality holds if

    $$a^n\geq max(2^{3.2277n},2^{1.9506n}\cdot b^{3n/7})$$ $$max(192\cdot a^{98/27}/437500000000, 48\cdot a^{16/9}\cdot b^{2/3}/875000000)=C_4.$$\\
    
\textbf{Case 5}
$\ceil{17n/27}\leq s\leq \ceil{18n/27}$

Either two different subbrays have disappeared, or the two rays got reduced to a subset of $B_{\ceil{n/3}}$, and to a subset of $S_{\ceil{n/9}}\cup S_{\ceil{n/27}}$ respectively.
The two entire subrays could miss from two separate rays or from the same ray. This gives $3$ cases.

$$|T_5|\leq 3\binom{\ceil{18n/27}}{\floor{n/3}}\cdot \binom{\ceil{2n/3}}{\floor{n/3}}\leq 12\cdot 2^{2n/3+2n/3}/49\leq12\cdot 2^{4n/3}/49$$

When one ray is a subset of a copy of $S_{\ceil{n/9}}$ and of $S_{\ceil{n/27}}$, by Lemma \ref{decomp} and the estimation lemmas,  we get an upper bound of $\binom{\ceil{3n/27}+\ceil{n/27}}{\floor{{3n/27}}}\cdot g(S_{\ceil{n/9}})\cdot g(S_{\ceil{n/27}})\leq 2\cdot 2^{4n/27}\cdot g(S_{\ceil{n/9}})\cdot g(S_{\ceil{n/27}})$ for the number of peeling sequences of the ray.

Thus,
\begin{equation*}
    \begin{split}
        \sigma_5&\leq |T_5|\cdot max[g(B_{\ceil{n/3}})\cdot2\cdot 2^{4n/27}\cdot g(S_{\ceil{n/9}})\cdot g(S_{\ceil{n/27}}),\\ & g(S_{\ceil{n/3}})\cdot g(S_{\ceil{n/9}}),4\cdot2^{4n/9}\cdot g(S_{\ceil{n/9}})^4] \\ &\leq (12\cdot 2^{4n/3}/49)\cdot max[(b^{n/3}/500)\cdot 2^{4n/27}(a^{n/9}/500)\cdot(a^{n/27}/500)\cdot 2\cdot a^{50/27}\cdot b^{2/3},\\ &(a^{n/3}/500)\cdot(a^{n/9}/500)\cdot a^{14/9},2^{4n/9}\cdot (a^{n/9}/500)^4\cdot 4a^{32/9}]\\ &\leq C_5\cdot a^n
    \end{split}
\end{equation*}

    The last inequality holds if

    $$a^n\geq max(2^{40n/23}\cdot b^{9n/23}, 2^{12n/5}, 2^{16n/5})$$ 
    $$max[24\cdot a^{50/27}\cdot b^{2/3}/612500, 12\cdot a^{14/9}/1225, 48\cdot a^{32/9}/306250000]\cdot 10^{-4}=C_5.$$ \\
    
\textbf{Case 6}
$\ceil{18n/27}\leq s\leq \ceil{22n/27}$

Two subrays have disappeared.
The two subrays could miss from two separate rays or from the same ray, giving two cases.
In the case when the subrays are missing from separate rays, we also know that one of the rays also misses another sub-subray (an $n/27$ piece) by the pigeonhole principle. For all these kinds of ray structure, we have our estimates from previous cases.

$$|T_6|\leq 3\binom{\ceil{22n/27}}{\ceil{n/3}}\cdot \binom{\ceil{2n/3}}{\floor{n/3}}\leq 24\cdot 2^{22H(9/22)n/27+2n/3}/7\leq24\cdot2^{1.4620n}/7$$

Thus,

\begin{equation*}
    \begin{split}
            \sigma_6&\leq |T_6|\cdot max[g(S_{\ceil{n/3}})\cdot g(S_{\ceil{n/9}}),\\ & 2\cdot 2^{2n/9}\cdot g(S_{\ceil{n/9}})^2\cdot8\cdot 2^{7n/27}\cdot g(S_{\ceil{n/9}})\cdot g(S_{\ceil{n/27}})^2] \\ &\leq (24\cdot2^{1.4620n}/7)\cdot max[(a^{n/3}/500)\cdot (a^{n/9}/500)\cdot a^{14/9},\\ &2^{13n/27}\cdot(a^{n/9}/500)^3\cdot (a^{n/27}/500)^2\cdot 16\cdot a^{98/27}]\\ &\leq C_6\cdot a^n
    \end{split}
\end{equation*}

    The last inequality holds if

    $$a^n\geq max(2^{2.26317n},2^{3.28n})$$ $$max(24\cdot a^{14/9}/1750000, 384\cdot a^{98/27}/437500000000)=C_6.$$\\
    
\textbf{Case 7}
$\ceil{22n/27}\leq s\leq \ceil{27n/27}$

At least three different subrays have disappeared.
Two subrays miss from one ray, and one subray misses from the other one. 

$$|T_7|\leq 3\binom{\ceil{27n/27}}{\ceil{n/3}}\cdot \binom{\ceil{2n/3}}{\floor{n/3}}\leq 24\cdot 2^{27H(9/27)n/3+2n/3}/7\leq24\cdot2^{1.5850n}/7$$

Thus,
\begin{equation*}
\begin{split}
    \sigma_7&\leq |T_7|\cdot (g(S_{\ceil{n/9}})\cdot2\cdot2^{2n/9}\cdot g(S_{\ceil{n/9}})^2)\\ &\leq (24\cdot2^{1.5850n}/7)\cdot (a^{n/9}/500)^3\cdot2^{2n/9}\cdot2\cdot a^{24/9}\\ &\leq C_7\cdot a^n
    \end{split}
\end{equation*}

    The last inequality holds if

        $$a^n\geq 2^{2.7109n}$$ $$48\cdot a^{24/9}/875000000=C_7.$$\\

    Collecting all constraints, we need to verify that $a = 9.78$ satisfies each inequality.

$a^{n}\geq9.76^n\geq 2^{3.2865n}$

$a^n\geq9.78^n\geq 2^{1.7315n}\cdot b^{n/2}$

$a^n\geq9.74^n\geq max(2^{3.2138n},2^{1.2061n}\cdot b^{2n/3})$

$a^n\geq9.76^n\geq max(2^{3.2277n},2^{1.9506n}\cdot b^{3n/7})$

$a^n\geq9.19^n\geq max(2^{40n/23}\cdot b^{9n/23}, 2^{12n/5}, 2^{16n/5})$

$a^n\geq9.72^n\geq max(2^{2.26317n},2^{3.28n})$

    $a^n\geq6.55^n\geq 2^{2.7109n}$

We also need to check the sum of the coefficients:

$C_1=24\cdot a^{4/3}/1750000\leq 0.00029$

$C_2=24\cdot a^{2/3}\cdot b^{2/3}/1750000\leq 0.00027$

$C_3=max(48\cdot a^{22/9}/875000000,24\cdot b^{4/3}/1750000)\leq0.00025$

$C_4=max(192\cdot a^{98/27}/437500000000, 48\cdot a^{16/9}\cdot b^{2/3}/875000000)\leq0.00002$

$C_5=max(24\cdot a^{\frac{50}{27}}\cdot b^{\frac{2}{3}}/6125, 12\cdot a^{\frac{14}{9}}/12.25, 48\cdot a^{\frac{32}{9}}/3062500)\cdot 10^{-6}\leq 4\cdot 10^{-5}$

$C_6=max(24\cdot a^{14/9}/1750000, 384\cdot a^{98/27}/437500000000)\leq 0.00048$

$C_7=48\cdot a^{24/9}/875000000\leq0.00003$

$\sum_{i=1}^7C_i\leq 0.00138<0.002=\frac{1}{500}$ which is sufficient, and concludes the proof of the theorem.
\end{proof}

    \section{Concluding Remarks}

  Despite considerable effort, the author was unable to improve the trivial lower bound on the minimal number of peeling sequences, this therefore remains an open problem. The intuition behind the upcoming conjecture comes from the Fractional Erdős-Szekeres Theorem \cite{barany}, which implies that for any $n$ there are $\alpha n$ many points, such that their convex hull forms a $4$-gon for many steps.

    \begin{conjecture}
        The minimal number of peeling sequences of $n$ points in $\R^d$ is at least $c(d+1+\epsilon)^n$ for some positive constants $c$ and $\epsilon$. So $g_d(n)=\Omega((d+1+\epsilon)^n)$.
    \end{conjecture}

    The upper bound can be improved with a more careful case analysis, for example by considering subrays of greater depth. It would also be interesting to find a different way of deriving a strong upper bound on the number of peeling sequences of $S_n$. The construction appears optimal: If the next vertex is always chosen randomly from the convex hull, one would expect to only choose between $3$ different vertices for most of the steps.
\section*{Acknowledgements}
Research supported by ERC advanced grant no. 882971: Geoscape.

\FloatBarrier


\end{document}